\documentclass[11pt]{amsart}
\usepackage[T1]{fontenc}
\usepackage{lmodern}
\usepackage[a4paper,margin=28mm,headheight=14pt]{geometry}
\usepackage{amsmath,amssymb,mathtools,mathrsfs}
\usepackage{microtype}
\usepackage{needspace}
\usepackage[hidelinks]{hyperref}
\hypersetup{pdftitle={Local integral pinching and planarity estimates for mean curvature flow},pdfauthor={Liang Cheng}}
\allowdisplaybreaks[2]
\numberwithin{equation}{section}
\newtheorem{theorem}{Theorem}[section]
\newtheorem{proposition}[theorem]{Proposition}
\newtheorem{lemma}[theorem]{Lemma}
\newtheorem{corollary}[theorem]{Corollary}
\newtheorem*{langfordtheorem}{Theorem}
\theoremstyle{remark}
\newtheorem{remark}[theorem]{Remark}
\newcommand{\R}{\mathbb R}

\newcommand{\dmu}{\,d\mu_t}
\newcommand{\dmt}{\,d\mu_t\,dt}
\newcommand{\cE}{\mathcal E}
\newcommand{\cC}{\mathfrak C}
\newcommand{\cQ}{\mathcal Q}
\newcommand{\spt}{\operatorname{spt}}
\newcommand{\tr}{\operatorname{tr}}

\newcommand{\Avec}{\mathbf A}
\newcommand{\HH}{\mathbf H}
\newcommand{\Ah}{\widehat{\mathbf A}}
\newcommand{\nablaN}{\nabla^{\perp}}
\title[Local pinching estimates]{Local Convexity and Planarity Estimates for Mean Curvature Flow}
\author{Liang Cheng}
	\address{School of Mathematics and Statistics  $\&$ Key Laboratory of Nonlinear Analysis and Applications (Ministry of Education), Central  China Normal University, Wuhan, 430079, P.R.China}

	\email{chengliang@ccnu.edu.cn }
	
\thanks{Liang Cheng's Research partially supported by  
	National Natural Science Foundation of China 12171180}


\begin{document}

\begin{abstract}
We establish local convexity and cylindrical pinching estimates for hypersurface mean curvature flow under an intermediate convexity condition and integral bounds on the mean curvature. These estimates hold after a positive waiting time and do not require pointwise upper bounds for the mean curvature on the parabolic boundary. In arbitrary codimension, we also establish local planarity estimates for properly defined flows satisfying a uniform $c_n$-pinching condition, without assuming bounded curvature. As applications, the hypersurface estimates yield convexity and cylindrical pinching for complete, spacetime-proper ancient hypersurface flows with uniform Euclidean area growth, without any bounded-curvature assumption. The planarity estimates imply that every complete, connected, spacetime-proper ancient flow satisfying the uniform $c_n$-pinching condition is either a static affine $n$-plane or lies in a fixed affine $(n+1)$-plane for all time.
\end{abstract}

\maketitle

\section{Introduction}

\subsection{Background and objectives}
Curvature pinching estimates play a fundamental role in the study
of singularities and ancient solutions of mean curvature flow.
They provide quantitative control of the second fundamental form
in regions of high curvature and are essential for understanding
the geometry of singularity models.
For hypersurface mean curvature flow, classical global results
include Huisken's umbilic estimate for convex hypersurfaces
\cite{Huisken84}, the convexity estimates of Huisken--Sinestrari
for mean-convex hypersurfaces \cite{HS99}, and their cylindrical
estimates for two-convex hypersurfaces \cite{HS09}.
These estimates show that the geometry of high-curvature regions
is increasingly constrained by convexity and cylindrical pinching.

A common feature of these global estimates is that their
constants depend on the initial hypersurface. Since singularity
formation is essentially a local phenomenon, it is desirable
to obtain estimates depending only on geometric data in a
neighborhood of the region under consideration.
Motivated by this, Langford \cite{Langford23} introduced a local
version of the Huisken--Stampacchia iteration and derived
local convexity and cylindrical pinching estimates for properly
immersed codimension-one mean curvature flows. His estimates require a curvature
pinching condition together with pointwise and integral curvature
bounds on an initial time slice and a lateral annulus.

In this paper, we first establish local convexity and cylindrical pinching estimates similar to those of Langford \cite{Langford23}. Building on the ideas of Langford \cite{Langford23}, we work under integral curvature bounds without imposing a pointwise upper bound on the mean curvature. We then apply these estimates to derive convexity and cylindrical pinching results for spacetime-proper ancient hypersurface flows with uniform Euclidean area growth. Separately, we also study the local pinching estimates for the mean curvature flow of higher codimensions. Using the geometric differential
inequalities of Naff and Lee--Naff--Zhu \cite{Naff22,LNZ25}, we
establish local pinching estimates, leading to
codimension reduction for spacetime-proper ancient flows without
assuming polynomial curvature growth.

\subsection{local convexity estimates in the codimension-one case}
Let \(X:M^n\times I\to\R^{n+1}\), \(n\ge2\), be a smooth mean 
curvature flow on the hypersurface equipped with a smooth unit normal \(\nu\). The shape
operator is \(A=D\nu\), and we adopt the sign convention
\[
 \partial_tX=-H\nu,\qquad
 \kappa_1\le\cdots\le\kappa_n,\qquad
 H=\sum_{i=1}^n\kappa_i,\qquad |A|^2=\sum_{i=1}^n\kappa_i^2,
\]
where \(\kappa_1,\cdots,\kappa_n\) are the principal curvatures.
For \(x_0\in\R^{n+1}\) and \(r>0\), let
\[
 B_r(x_0):=\{x\in\R^{n+1}:|x-x_0|<r\}.
\]
For a spacetime point \(z_0=(x_0,t_0)\in\R^{n+1}\times\R\), we write
\[
 P_r(z_0)=B_r(x_0)\times(t_0-r^2,t_0].
\]

In \cite{Langford23}, Langford introduced a local
version of the Huisken--Stampacchia iteration and derived the following
local convexity and cylindrical pinching estimates for properly
immersed codimension-one mean curvature flows. 

\begin{langfordtheorem}[Langford \cite{Langford23}]
Let $n\ge2$, $m\in\{0,\ldots,n-1\}$, $\alpha>0$, $q>0$,
$\lambda>\delta>0$, and $\Theta,V>0$ satisfy
\[
 \sqrt q\,\Theta\ge(\lambda-\delta)^{-1},\qquad \Theta V\ge1.
\]
Fix $R>0$, and set
$T=R^2/(2n)$ and
$\mathcal A_R=B_{\lambda R}\setminus B_{(\lambda-\delta)R}$.
Suppose a smooth mean curvature flow is properly defined in
$\overline{B_{\lambda R}}\times[0,T)$ and satisfies
\[
 \kappa_1+\cdots+\kappa_{m+1}\ge\alpha|A|
 \quad\text{on }(\overline{B_{\lambda R}}\times\{0\})
      \cup(\partial B_{\lambda R}\times(0,T)),
\]
\begin{equation}\label{eq:point_upbound_H}
   \sup_{(\overline{B_{\lambda R}}\times\{0\})\cup
             (\mathcal A_R\times(0,T))}\frac Hn
 \le\Theta R^{-1},  
\end{equation}
and
\[
 T\int_{M_0\cap B_{\lambda R}}H^q\,d\mu_0
 +\delta^{-2}\int_0^T\int_{M_t\cap\mathcal A_R}H^q\dmt
 \le (\Theta R^{-1})^q(VR)^{n+2}.
\]
Then, for every $\varepsilon>0$ and $\vartheta\in(0,1)$,
\begin{align*}
 \kappa_1&\ge-\varepsilon H-C_{\epsilon}R^{-1},\\
 \sum_{j=m+1}^n(\kappa_n-\kappa_j)
 &\le\sum_{j=1}^m\kappa_j+\varepsilon H+C_{\epsilon}R^{-1}
\end{align*}
on $B_{\vartheta(\lambda-\delta)R}\times(0,T)$, where
\[
 C_{\epsilon}
 =c(n,\alpha,q,\varepsilon)\Theta
             \left(\frac{\Theta V}{1-\vartheta}\right)^{2/q}.
\]
\end{langfordtheorem}

Integral bounds for the mean curvature arise naturally under various
geometric assumptions. For instance, on a closed, strictly convex
hypersurface satisfying $\kappa_1\ge\varepsilon H>0$, the Gauss map
is a diffeomorphism onto $\mathbb S^n$. Since
$\det A\ge\varepsilon^n H^n$, the area formula yields
\[
\int_{M_t} H^n\,d\mu_t
\le \varepsilon^{-n}|\mathbb S^n|.
\]
Another natural source of integral curvature control is a bound on
the spacetime area, which yields a local $L^2$-estimate for $H$;
see Lemma~\ref{lem:areaDissipation} below.
In contrast, a pointwise upper bound for the mean curvature is
generally much harder to obtain, particularly in local or noncompact
settings. This motivates our first main result, which establishes
local convexity and cylindrical pinching estimates without the
pointwise upper bound for mean curvature (i.e. \eqref{eq:point_upbound_H}) required in Langford's theorem,
assuming only suitable integral curvature bounds and allowing
a positive waiting time.

\Needspace{12\baselineskip}
\begin{theorem}
\label{thm:main}
Fix $n\ge2$, $m\in\{0,\ldots,n-1\}$, $\alpha>0$, $q>0$,
$\lambda>\delta>0$, and $R>0$, and set $T=R^2/(2n)$.
Suppose that a smooth mean curvature flow is properly defined in
$\overline B_{\lambda R}\times[0,T)$ and satisfies
\begin{equation}\label{eq:cone}
 \kappa_1+\cdots+\kappa_{m+1}\ge\alpha|A|
 \quad\text{on }(\overline{B_{\lambda R}}\times\{0\})
      \cup(\partial B_{\lambda R}\times(0,T)).
\end{equation}
Let $\mathcal A_R=B_{\lambda R}\setminus B_{(\lambda-\delta)R}$ and define
\begin{equation}\label{eq:E}
 \cE_q=R^{q-n-2}\left\{
 T\int_{M_0\cap B_{\lambda R}}H^q\,d\mu_0
 +\delta^{-2}\int_0^T\int_{M_t\cap\mathcal A_R}H^q\dmt\right\}.
\end{equation}
Assume that $\cE_q<\infty$. Given $\varepsilon>0$,
$\vartheta\in(0,1)$, and $\tau\in(0,1/(2n))$, the estimates
\begin{align}
 \kappa_1&\ge-\varepsilon H-C_{\varepsilon,\tau}R^{-1},
 \label{eq:convexity}\\
 \sum_{j=m+1}^n(\kappa_n-\kappa_j)
 &\le\sum_{j=1}^m\kappa_j+\varepsilon H+C_{\varepsilon,\tau}R^{-1}
 \label{eq:cylindrical}
\end{align}
hold in $B_{(\lambda-\delta)\vartheta R}\times[\tau R^2,T)$, where
\begin{equation}\label{eq:D}
 \begin{gathered}
  D_{\tau}=1+\tau^{-1/2}+\frac{1}{(1-\vartheta)(\lambda-\delta)},\\
  C_{\varepsilon,\tau}=c(n,\alpha,q,\varepsilon)D_{\tau}
  \left[1+\left(D_{\tau}^{n+2-q}\cE_q\right)^{\frac{2}{q(n+2)}}\right].
 \end{gathered}
\end{equation}
\end{theorem}

For ease of application, we record the following consequence of Theorem~\ref{thm:main}.

\Needspace{12\baselineskip}
\begin{corollary}\label{cor:collar}
Fix $n\ge2$, $m\in\{0,\ldots,n-1\}$, $\alpha>0$, $q>0$,
$\theta\in(1,2)$, $R>0$, and $z_0=(x_0,t_0)\in\R^{n+1}\times\R$.
Suppose a smooth mean curvature flow is properly defined in
$P_{2R}(z_0)$ and satisfies
\[
 \kappa_1+\cdots+\kappa_{m+1}\ge\alpha|A|
\]
throughout $P_{2R}(z_0)$.
Define
\begin{equation}\label{eq:collarSet}
 \mathcal C_{\theta,R}(z_0):=P_{2R}(z_0)\setminus P_{\theta R}(z_0).
\end{equation}
Here, \[
 P_r(z_0):=B_r(x_0)\times(t_0-r^2,t_0].
\]
Set
\begin{equation}\label{eq:collarQuantity}
 \cC_{q,\theta}(z_0;R)
 :=R^{q-n-2}\iint_{\mathcal C_{\theta,R}(z_0)}H^q\dmt.
\end{equation}
If $\cC_{q,\theta}(z_0;R)\le K<\infty$, then for every $\varepsilon>0$
there is $C=C(n,\alpha,q,\varepsilon,\theta)>0$ such that
\begin{align}
 \kappa_1&\ge-\varepsilon H-
       \frac{C}{R}\left(1+K^{2/(q(n+2))}\right),\label{eq:interiorConvexity}\\
 \sum_{j=m+1}^{n}(\kappa_n-\kappa_j)
 &\le\sum_{j=1}^{m}\kappa_j+\varepsilon H+
       \frac{C}{R}\left(1+K^{2/(q(n+2))}\right)\label{eq:interiorCylindrical}
\end{align}
on $P_R(z_0)$.
\end{corollary}

The uniform Euclidean area-growth bound controls the $L^2$-energy of the mean curvature under the mean curvature flow; see Lemma~\ref{lem:areaDissipation} below. As an application of Corollary~\ref{cor:collar}, we therefore obtain the following corollary.

\begin{theorem}\label{thm:areaAncient}
	Let
$
	X:M^n\times(-\infty,0)\to\R^{n+1}
$
	be a smooth, complete ancient mean curvature flow without
	boundary, and suppose that the spacetime map $(X,t)$ is
	proper. Assume that, for some
	$m\in\{0,\ldots,n-1\}$ and $\alpha>0$,
	\[
	\kappa_1+\cdots+\kappa_{m+1}\ge\alpha|A|
	\]
	throughout $M\times(-\infty,0)$. Suppose further that
	there exists $\Lambda<\infty$ such that
	\begin{equation}\label{eq:uniformAreaGrowth}
		\mu_t(M_t\cap B_r(y))\le\Lambda r^n
		\qquad
		(t<0,\ y\in\R^{n+1},\ r>0).
	\end{equation}
	Then
	\begin{equation}\label{eq:areaAncientCone}
		\kappa_1\ge0,\qquad
		(n-m)\kappa_n\le H,\qquad
		|A|^2\le\frac{H^2}{n-m}
	\end{equation}
	on $M\times(-\infty,0)$. 	
	In particular, an ancient mean curvature flow $
	X:M^n\times(-\infty,0)\to\R^{n+1}
	$ satisfying the completeness, properness and
	volume hypotheses \eqref{eq:uniformAreaGrowth}, together with
	\begin{equation}\label{eq:entropyWeakPinching}
		|A|\le LH
		\qquad\text{for some }0<L<\infty
	\end{equation}
	is weakly convex. 
\end{theorem}

\begin{remark}
Lee--Naff--Zhu \cite[Theorem~1.3]{LNZ25} established
	convexity for mean-convex ancient hypersurface flows
	satisfying 
	\begin{equation}\label{eq:uniformEntropy}
		\Lambda_{\mathrm{ent}}
		:=\sup_{t<0}\sup_{y\in\R^{n+1},\,r>0}
		(4\pi r^2)^{-n/2}
		\int_{M_t}
		\exp\left(-\frac{|X-y|^2}{4r^2}\right)\dmu<\infty.
	\end{equation}
	 and the uniform
	curvature pinching condition \eqref{eq:entropyWeakPinching},
	under an additional polynomial curvature-growth assumption.
	They asked in \cite[Question~5.3]{LNZ25} whether this polynomial curvature-growth
 assumption can be removed.
	Observe that the entropy bound \eqref{eq:uniformEntropy}
	implies the uniform  area-growth estimate
	\begin{equation}\label{eq:entropyToArea}
		\mu_t(M_t\cap B_r(y))
		\le e^{1/4}(4\pi)^{n/2}
		\Lambda_{\mathrm{ent}}\,r^n
	\end{equation}
	for all $t<0$, $y\in\R^{n+1}$, and $r>0$.
	Consequently, Theorem~\ref{thm:areaAncient} yields
	convexity without the polynomial curvature-growth
	assumption, provided the flow is spacetime-proper.
	This gives an affirmative answer to the question of
	Lee--Naff--Zhu within the class of spacetime-proper
	ancient flows.
	
\end{remark}

\subsection{Local convexity estimates in the higher-codimension case}

Let $F:M^n\times I\to\R^N$, where $n\ge2$ and $N\ge n+1$,
be a smooth mean curvature flow satisfying
\[
\partial_tF=\HH=\Delta F,
\]
where $\Avec$ denotes the normal-valued second fundamental form
and $\HH=\tr_g\Avec$ is the mean curvature vector.
Set
\begin{equation}\label{hc:eq:cn}
	c_n=\min\left\{
	\frac4{3n},\frac{3(n+1)}{2n(n+2)}
	\right\}.
\end{equation}
We consider the following   pinching condition
\begin{equation}\label{hc:eq:pinch}
	|\Avec|^2\le(c_n-\delta_0)|\HH|^2
\end{equation}
for some fixed $\delta_0\in(0,c_n)$.
At points where $|\HH|>0$, let
\begin{equation}\label{hc:eq:decomposition}
	\nu_1=\frac{\HH}{|\HH|},\qquad
	h_{ij}=\langle\Avec_{ij},\nu_1\rangle,\qquad
	\Ah_{ij}=\Avec_{ij}-h_{ij}\nu_1.
\end{equation}
Thus
\[
\tr_g h=|\HH|,
\qquad
|\Avec|^2=|h|^2+|\Ah|^2.
\]

For higher codimensional mean curvature flow, Naff~\cite{Naff22} established a planarity estimate for compact mean curvature flows under the pinching condition $|\Avec|^2 \leq c_n|\HH|^2$. Lee--Naff--Zhu~\cite{LNZ25} subsequently obtained the planarity estimate under the pinching condition \eqref{hc:eq:pinch} and established codimension reduction for pinched ancient flows under an additional polynomial curvature-growth assumption.

Building on their geometric differential inequalities, we establish a local planarity estimate without a pointwise curvature bound. This also yields a time-dependent estimate and, for spacetime-proper ancient flows, codimension reduction without assuming polynomial curvature growth.
Our first result in this direction is the following local planarity estimate.

\begin{theorem}\label{hc:thm:local}
Fix $n\ge2$, $N\ge n+1$, $\delta_0\in(0,c_n)$, $R>0$,
$y\in\R^N$ and $t_0\in\R$. Suppose a smooth mean curvature flow
$F:M^n\times I\to\R^N$, with $M$ without boundary and
$[t_0-R^2,t_0]\subset I$, is properly defined in
\[
 \overline{B_R(y)}\times[t_0-R^2,t_0]
\]
and satisfies \eqref{hc:eq:pinch} on
$B_R(y)\times[t_0-R^2,t_0]$. Then
\begin{equation}\label{hc:eq:local}
 |\Ah|^2\le C R^{-2\sigma_0}|\HH|^{2-2\sigma_0}
\end{equation}
on $B_{R/2}(y)\times[t_0-R^2/2,t_0]$, where
$\sigma_0=\sigma_0(n,\delta_0)\in(0,1)$ and $C=C(n,\delta_0)>0$.
\end{theorem}

As a direct consequence of Theorem~\ref{hc:thm:local}, we obtain the
following estimate without the bounded-curvature assumption imposed
in \cite[Theorem~1.1]{LNZ25}.

\begin{theorem}\label{hc:cor:timePlanarity}
Fix $n\ge2$, $N\ge n+1$, $\delta_0\in(0,c_n)$ and $T>0$.
Let
\[
 F:M^n\times[0,T]\longrightarrow\R^N
\]
be a smooth, complete mean curvature flow, where $M$ is nonempty,
connected and without boundary. Suppose that 
$F$ is proper and that
\eqref{hc:eq:pinch} holds throughout $M\times[0,T]$ for some $\delta_0>0$.
Then either the flow is a static affine $n$-plane, or $|\HH|>0$ on
$M\times(0,T]$ and
\begin{equation}\label{hc:eq:timePlanarity}
 |\Ah|^2\le C t^{-\sigma_0}|\HH|^{2-2\sigma_0}
 \qquad\text{on }M\times(0,T],
\end{equation}
where $\sigma_0=\sigma_0(n,\delta_0)\in(0,1/2)$ and
$C=C(n,\delta_0)>0$. 
\end{theorem}

As a further consequence of Theorem~\ref{hc:cor:timePlanarity},
we obtain the following codimension-reduction result for
spacetime-proper ancient flows.

\begin{corollary}\label{hc:thm:ancient}
Let $n\ge2$ and $N\ge n+1$. Let
\[
 F:M^n\times(-\infty,0)\longrightarrow\R^N
\]
be a smooth, complete ancient mean curvature flow, where $M$ is nonempty,
connected and without boundary. Suppose that 
$F$ is proper and that
\eqref{hc:eq:pinch} holds throughout $M\times(-\infty,0)$ for some $\delta_0>0$. Then either the flow is a static affine $n$-plane,
or $|\HH|>0$ everywhere and there is a unique affine $(n+1)$-plane
$\Pi\subset\R^N$ such that
\begin{equation}\label{hc:eq:fixedplane}
 F(M,t)\subset\Pi\qquad(t<0).
\end{equation}
\end{corollary}

The paper is organized as follows. Section~\ref{sec:prelim}
collects the preliminary results. Section~\ref{sec:theoremProof}
proves Theorem~\ref{thm:main} and Corollary~\ref{cor:collar}.
Section~\ref{sec:areaEntropy} proves Theorem~\ref{thm:areaAncient}.
Finally, Section~\ref{sec:planarity} proves
Theorems~\ref{hc:thm:local} and~\ref{hc:cor:timePlanarity},
and Corollary~\ref{hc:thm:ancient}.

\section{Preliminaries}\label{sec:prelim}
In this section we gather several preliminary results that will be used below. We begin by recording the geometric Poincaré inequality.
\newcommand{\PoincareRef}{\cite[Proposition~2.7]{Langford17}}
\newcommand{\StampacchiaRef}{\cite{Stampacchia66}}

\begin{proposition}[Geometric Poincar\'e inequality \PoincareRef]
\label{prop:poincare}
Let $\Gamma\subset\R^n$ be an open symmetric cone satisfying
\[
 \overline\Gamma\setminus\{0\}
 \subset\{z\in\R^n:z_1+\cdots+z_n>0\},
 \qquad \overline\Gamma\cap\mathrm{Cyl}=\varnothing,
\]
where $\mathrm{Cyl}\subset\R^n$ denotes the union of the nonzero cylindrical
rays: up to permutation, their elements have $k$ zero entries and $n-k$
equal positive entries, for some $0\le k\le n-1$.
Let $X:M^n\to\R^{n+1}$ be a smooth hypersurface immersion and
$u\in C_c^1(M)$ a nonnegative function such that the curvature vectors
on $\spt u$ form a compact subset of $\Gamma$. Then, for every $\beta>0$,
\begin{equation}\label{eq:poincare}
 \int_M |A|^2u^2\,d\mu
 \le\beta\int_M|\nabla u|^2\,d\mu
 +P(1+\beta^{-1})\int_M\frac{|\nabla A|^2}{H^2}u^2\,d\mu,
\end{equation}
where $P=P(n,\Gamma)>0$. 
\end{proposition}

We also use the parabolic consequence of the Michael--Simon inequality
\cite{MS73}. Let $J$ be a finite time interval. For a spatially compactly
supported function $w$ with $w(\cdot,t)\in W^{1,2}_0(M_t)$ for almost
every $t$ and with finite right-hand side below,
\begin{align}
 \left(\int_J\int |w|^{2(n+2)/n}\dmt\right)^{n/(n+2)}
 &\le C(n)
 \left(\mathop{\rm ess\,sup}_{t\in J}\int w^2\dmu\right)^{2/(n+2)}
 \notag\\
 &\quad\times
 \left(\int_J\int(|\nabla w|^2+H^2w^2)\dmt\right)^{n/(n+2)}
 \notag\\
 &\le C(n)\left[
 \mathop{\rm ess\,sup}_{t\in J}\int w^2\dmu
 +\int_J\int(|\nabla w|^2+H^2w^2)\dmt\right].
 \label{eq:sobolev}
\end{align}
For $n\ge3$, this follows from the spatial $L^{2n/(n-2)}$ inequality
by interpolation with $L^2$ and integration in time. For $n=2$, the
$L^1$ Michael--Simon inequality applied to $w^2$ gives
\begin{align*}
 \left(\int w^4\dmu\right)^{1/2}
 &\le C\int(2|w||\nabla w|+|H|w^2)\dmu\\
 &\le C\left(\int w^2\dmu\right)^{1/2}
 \left(\int(|\nabla w|^2+H^2w^2)\dmu\right)^{1/2}.
\end{align*}
Squaring and integrating yields \eqref{eq:sobolev}. Approximation
extends these calculations from smooth functions to the stated class.

\begin{lemma}[Stampacchia iteration; cf.\ \StampacchiaRef]
\label{lem:stampacchia}
Let $p>0$, $\omega>0$, $b>1$ and $C_*>0$. Suppose $a(k,j)\ge0$ is
nonincreasing in $k\ge k_0$ and $j\in\mathbb N\cup\{0\}$, and
\begin{equation}\label{eq:stampacchiaRecurrence}
 a(h,j+1)\le C_*b^j(h-k)^{-p}a(k,j)^{1+\omega}
\end{equation}
for every $h>k\ge k_0$ and $j\ge0$. If $a(k_0,0)\le a_0<\infty$,
then there is $C=C(p,\omega,b)>0$ such that, for every $d>0$,
\begin{equation}\label{eq:stampacchiaCondition}
 d^p\ge C C_*a_0^\omega
\end{equation}
implies $\lim_{j\to\infty}a(k_0+d,j)=0$.
\end{lemma}

\begin{proof}
The assertion is immediate if $a_0=0$. Otherwise put
$k_j=k_0+d(1-2^{-j})$ and $B=2^pb$. The recurrence gives
\[
 a(k_{j+1},j+1)\le 2^pC_*d^{-p}B^j a(k_j,j)^{1+\omega}.
\]
If $2^pC_*d^{-p}a_0^\omega\le B^{-1/\omega}$, induction yields
$a(k_j,j)\le a_0B^{-j/\omega}$. Monotonicity in $k$ then proves
the conclusion. Thus $C=2^pB^{1/\omega}$ suffices.
\end{proof}

For the planarity estimates, use the normal-valued notation introduced
above. The notation $\nablaN$ includes the induced connections on
normal-valued tensors.\label{hc:sec:geometricInput}

\begin{lemma}[Local pinching inequalities
 {\cite[Lemma~2.4 and equation~(2.6), with Appendix~B]{LNZ25}}]
\label{hc:lem:input}
 Let $F:M^n\times I\to\R^N$, $n\ge2$ and
$N\ge n+1$, be a smooth mean curvature flow.
For $\delta_0\in(0,c_n)$,  there exist
$c_0=c_0(n,\delta_0)$ and $\sigma_0=\sigma_0(n,\delta_0)$ satisfying
\begin{equation}\label{hc:eq:parameters}
 \max\left\{\frac1n,c_n-\delta_0\right\}<c_0<c_n,
 \qquad 0<\sigma_0\le\frac18,
\end{equation}
with the following property. Suppose that \eqref{hc:eq:pinch} holds, and set
\begin{equation}\label{hc:eq:f}
 f=c_0|\HH|^2-|\Avec|^2.
\end{equation}
Then
\begin{equation}\label{hc:eq:fBounds}
 (c_0-c_n+\delta_0)|\HH|^2\le f\le c_0|\HH|^2,
\end{equation}
and $(\partial_t-\Delta)f\ge0$ throughout the region. On $\{f>0\}$,
\begin{align}
 (\partial_t-\Delta)f
 &\ge 2\left(|h|^2+\frac{|\Ah|^2}{nc_0-1}\right)f
 +2\left(1-\frac{c_0(n+2)}3\right)|\nablaN\Avec|^2,
 \label{hc:eq:Lf}\\
 (\partial_t-\Delta)|\Ah|^2
 &\le(1-2\sigma_0)\frac{|\Ah|^2}{f}(\partial_t-\Delta)f.
 \label{hc:eq:La}
\end{align}
\end{lemma}


\section{Proof of Theorem~\ref{thm:main} and Corollary \ref{cor:collar}}\label{sec:theoremProof}

\medskip

We briefly outline the proof of Theorem~\ref{thm:main}.
The proof follows Langford's localized Huisken--Stampacchia
scheme, in which the geometric Poincar\'e inequality is used
to absorb the unfavorable curvature reaction term in the
evolution of the pinching quantity $G_{\varepsilon,\sigma}$,
leading to a localized energy estimate and a Stampacchia
iteration. In Langford's argument, the higher moments needed
for this iteration are controlled using a pointwise upper
bound for $H$ on the initial slice and lateral annulus.

Our main modification is to obtain these higher moments
without such a pointwise bound.
The main difference in our argument lies in the control of
the higher moments required for the Stampacchia iteration.
We retain the favorable $H^2v^2$ term arising from the
area evolution in the localized energy inequality
\eqref{eq:damped}. Choosing $\sigma p=q$ and using the
boundedness of the normalized pinching quantity
$G_{\varepsilon,0,+}$, we first obtain the lower-order
integral estimate \eqref{eq:low} from the prescribed
initial and lateral $L^q$ data. We then introduce a
spacetime cutoff vanishing near the initial time.
The resulting inequality \eqref{eq:gainstart} provides
a gain of two powers of $H$, which, together with
interpolation and absorption, yields the higher-moment
estimates \eqref{eq:gainH}--\eqref{eq:gain}. In particular,
these estimates provide the higher integrability needed
to control the curvature reaction term in the truncated
energy inequality \eqref{eq:truncated} and close the
Stampacchia iteration \eqref{eq:levelrecurrence}.
This replaces the pointwise boundary curvature bound
used in Langford's argument by an integral bootstrap,
at the expense of a positive waiting time.

\begin{proof}[Proof of Theorem~\ref{thm:main}]
	By parabolic scaling, we may assume $R=1$, so that $T=1/(2n)$.
	Write $R_0=\lambda-\delta$.

	Set
	$S_{m+1}=\kappa_1+\cdots+\kappa_{m+1}$. The standard evolution for the mean curvature flow gives
	\[
	(\partial_t-\Delta)(S_{m+1}-\alpha|A|)
	\ge |A|^2(S_{m+1}-\alpha|A|)
	\]
	in the weak barrier sense. Applying the maximum principle and using \eqref{eq:cone}, we obtain
	\[
	S_{m+1}\ge\alpha|A|
	\]
	throughout $B_\lambda\times[0,T)$. Thus either $A\equiv 0$, or
	$\kappa_{m+1}>0$ and all remaining principal curvatures are positive. Since the theorem is trivial in the case \(A\equiv 0\), we may assume henceforth that \(|A|>0\).
	In particular,
	\begin{equation}\label{eq:HcontrolsA}
		H\ge \alpha|A| >0.
	\end{equation}

	For the two estimates respectively, take
	\[
	G=-\kappa_1\qquad\text{or}\qquad
	G=\kappa_n-\frac1{n-m}H,
	\]
	and use the notations
	\begin{equation}\label{eq:Geps}
		\begin{gathered}
			G_{\varepsilon}=G-\varepsilon\left(H-\frac\alpha2|A|\right),
			\qquad G_{\varepsilon,\sigma}=G_{\varepsilon}H^{\sigma-1},\\
			G_{\varepsilon,\sigma,+}=\max\{G_{\varepsilon,\sigma},0\},
			\qquad 0<\sigma<1.
		\end{gathered}
	\end{equation}
	In particular,
	\[
	G_{\varepsilon,0,+}=\frac{(G_{\varepsilon})_+}{H}.
	\]
	Hence, we have
	\begin{equation}\label{eq:normalized}
		\frac H2\le H-\frac\alpha2|A|\le H,\qquad
		0\le G_{\varepsilon,0,+}\le c(n,\alpha),\qquad
		G_{\varepsilon,\sigma,+}=G_{\varepsilon,0,+}H^\sigma.
	\end{equation}
	
		The standard calculation
	(see, for example, \cite[Proposition~12.9]{ACGL20} and
	\cite[Section~3]{Langford17}) gives a constant
	$\gamma=\gamma(n,\alpha,\varepsilon)>0$ such that
	\begin{equation}\label{eq:diff}
		\frac{(\partial_t-\Delta)G_{\varepsilon,\sigma}}
		{G_{\varepsilon,\sigma}}
		\le \sigma|A|^2-\gamma\frac{|\nabla A|^2}{H^2}
		+\gamma^{-1}
		\frac{|\nabla G_{\varepsilon,\sigma}|^2}
		{G_{\varepsilon,\sigma}^2}.
	\end{equation}
	on $\{\,G_{\varepsilon,\sigma}>0\}$. Here and below, all differential inequalities are understood in the distributional sense.
	
	\smallskip\noindent\emph{\textbf{The $L^2$-estimate.}}
	Let $\psi$ be a standard ambient spacetime cutoff.
	Set
	\[
	v=G_{\varepsilon,\sigma,+}^{p/2}.
	\]
	For $p\ge6(1+\gamma^{-1})$, \eqref{eq:diff} and
	Young's inequality give
	\begin{align}
		(\partial_t-\Delta)(\psi^2v^2)
		&\le 2|\psi(\partial_t-\Delta)\psi|v^2
		+6|\nabla\psi|^2v^2
		-\frac43\psi^2|\nabla v|^2\notag\\
		&\quad-\gamma p\psi^2v^2\frac{|\nabla A|^2}{H^2}
		+\sigma p|A|^2\psi^2v^2.
		\label{eq:product}
	\end{align}
	It follows that
	\begin{align}
		\frac d{dt}\int\psi^2v^2\dmu
		&+\int\psi^2H^2v^2\dmu
		+\frac43\int\psi^2|\nabla v|^2\dmu
		+\gamma p\int\psi^2v^2\frac{|\nabla A|^2}{H^2}\dmu\notag\\
		&\le\sigma p\int\psi^2|A|^2v^2\dmu
		+\int\bigl(2|\psi(\partial_t-\Delta)\psi|
		+6|\nabla\psi|^2\bigr)v^2\dmu,
		\label{eq:rawenergy}
	\end{align}
	where we have used the
	area evolution of mean curvature flow
	\begin{equation}\label{eq:area}
		\partial_t d\mu_t=-H^2d\mu_t.
	\end{equation}
	Apply \eqref{eq:poincare} to $u=\psi v$ with $\beta=p^{-1/2}$ to obtain
	\begin{align}\label{eq:gpy2}
		\sigma p\int\psi^2|A|^2v^2\,d\mu
		&\le2\sigma p^{1/2}\int\psi^2|\nabla v|^2\,d\mu
		+2\sigma p^{1/2}\int|\nabla\psi|^2v^2\,d\mu\notag\\
		&\quad+\sigma pP(1+p^{1/2})
		\int\psi^2v^2\frac{|\nabla A|^2}{H^2}\,d\mu.
	\end{align}
	Choose
	\[
	L\ge\max\{2,\,6(1+\gamma^{-1})\},
	\qquad
	0<\ell\le\min\left\{\frac16,\frac{\gamma}{4P}\right\}.
	\]
	If
	\begin{equation}\label{eq:admissible}
		p\ge L,\qquad 0<\sigma<1,\qquad
		\sigma\le\ell p^{-1/2},
	\end{equation}
	then
	\[
	2\sigma p^{1/2}\le\frac13
	\]
	and
	\[
	\sigma pP(1+p^{1/2})
	\le 2\sigma p^{3/2}P
	\le 2\ell pP
	\le\frac{\gamma p}{2}.
	\]
	Combining \eqref{eq:rawenergy} and \eqref{eq:gpy2},
	we therefore obtain
	\begin{align}
		\frac{d}{dt}\int\psi^2v^2\,d\mu_t
		&+\int\psi^2H^2v^2\,d\mu_t
		+\int\psi^2|\nabla v|^2\,d\mu_t\notag\\
		&+\frac{\gamma p}{2}
		\int\psi^2v^2\frac{|\nabla A|^2}{H^2}\,d\mu_t
		\notag\\
		&\le
		\int\left(
		2|\psi(\partial_t-\Delta)\psi|
		+\frac{19}{3}|\nabla\psi|^2
		\right)v^2\,d\mu_t.
		\label{eq:absorbed}
	\end{align}
	Since
	\[
	|\nabla(\psi v)|^2
	\le 2\psi^2|\nabla v|^2
	+2|\nabla\psi|^2v^2,
	\]
	we have
	\[
	\int\psi^2|\nabla v|^2\,d\mu_t
	\ge
	\frac12\int|\nabla(\psi v)|^2\,d\mu_t
	-
	\int|\nabla\psi|^2v^2\,d\mu_t.
	\]
	Discarding the nonnegative term involving $|\nabla A|^2$
	in \eqref{eq:absorbed} yields
	\begin{align}
		\frac{d}{dt}\int\psi^2v^2\,d\mu_t
		&+\frac12\int|\nabla(\psi v)|^2\,d\mu_t
		+\int\psi^2H^2v^2\,d\mu_t\notag\\
		&\le
		\int\left(
		2|\psi(\partial_t-\Delta)\psi|
		+\frac{22}{3}|\nabla\psi|^2
		\right)v^2\,d\mu_t.
		\label{eq:damped}
	\end{align}
	
	Choose a time-independent ambient cutoff
	$\zeta\in C_c^\infty(B_\lambda)$ such that
	\[
	0\le\zeta\le1,\qquad
	\zeta=1\ \text{on }B_{R_0},\qquad
	|D\zeta|\le c\delta^{-1},\qquad
	|D^2\zeta|\le c\delta^{-2},
	\]
	where $R_0=\lambda-\delta$, and set $\psi=\zeta\circ X$. Then
	\begin{equation}\label{eq:cutoff}
		(\partial_t-\Delta)\psi
		=\partial_t\zeta-\tr_{TM_t}D^2\zeta .
	\end{equation}
	Hence,
	\[
	(\partial_t-\Delta)\psi
	=-\operatorname{tr}_{TM_t}D^2\zeta.
	\]
	Consequently,
	\[
	|\psi(\partial_t-\Delta)\psi|
	+|\nabla\psi|^2
	\le c(n)\delta^{-2}
	\mathbf{1}_{B_\lambda\setminus B_{R_0}}.
	\]
	Discarding the nonnegative terms on the left-hand side of
	\eqref{eq:damped}, we obtain
	\[
	\frac{d}{dt}\int\psi^2v^2\,d\mu_t
	\le
	c(n)\delta^{-2}
	\int_{B_\lambda\setminus B_{R_0}}v^2\,d\mu_t.
	\]
	Integrating from $0$ to $t<T$ yields
	\[
	\int_{B_{R_0}}v^2\,d\mu_t
	\le
	\int_{B_\lambda}v^2\,d\mu_0
	+c(n)\delta^{-2}
	\int_0^t\int_{B_\lambda\setminus B_{R_0}}
	v^2\,d\mu_\xi\,d\xi.
	\]
It follows that
	\begin{align*}
		\int_0^T\int_{B_{R_0}}v^2\,d\mu_t\,dt
		&\le
		T\int_{B_\lambda}v^2\,d\mu_0\\
		&\quad+
		c(n)\delta^{-2}
		\int_0^T(T-\xi)
		\int_{B_\lambda\setminus B_{R_0}}
		v^2\,d\mu_\xi\,d\xi.
	\end{align*}
	Since $T=1/(2n)$, we conclude that
	\begin{equation}\label{eq:lowgeneral}
		\int_0^T\int_{B_{R_0}}v^2\,d\mu_t\,dt
		\le
		\frac1{2n}\int_{B_\lambda}v^2\,d\mu_0
		+\frac{c(n)}{2n}\delta^{-2}
		\int_0^T\int_{B_\lambda\setminus B_{R_0}}
		v^2\,d\mu_t\,dt.
	\end{equation}
	
	We now choose $p$ and $\sigma$ such that $\sigma p=q$.
	Set
	\begin{equation}\label{eq:Lambda}
		\Lambda=
		\left(
		\frac1{2n}\int_{B_\lambda}H^q\,d\mu_0
		+\frac{100}{\delta^2}
		\int_0^T\int_{B_\lambda\setminus B_{R_0}}
		H^q\,d\mu_t\,dt
		\right)^{1/q}.
	\end{equation}
	By the definition of $\cE_q$,
	\[
	\cE_q\le\Lambda^q\le100\cE_q.
	\]
	On the other hand, \eqref{eq:normalized} implies
	\[
	v^2=G_{\varepsilon,\sigma,+}^{p}
	=G_{\varepsilon,0,+}^{p}H^{\sigma p}
	\le c(n,\alpha)^p H^q.
	\]
	Substituting this into \eqref{eq:lowgeneral}
	and using \eqref{eq:Lambda}, we obtain
	\begin{equation}\label{eq:low}
		\int_0^T\int_{B_{R_0}}
		G_{\varepsilon,\sigma,+}^{p}\,d\mu_t\,dt
		\le c\Lambda^q,
		\qquad \sigma p=q,
	\end{equation}
	where $c$ depends only on $n,\alpha$ and the chosen exponent $p$.

	\smallskip\noindent\emph{\textbf{Higher-order integral estimates.}}
	Set
	\[
	r_* =\vartheta R_0,\qquad
	r_0=\frac{1+\vartheta}{2}R_0,\qquad
	\cQ_0=B_{r_0}\times[\tau/2,T).
	\]
    \textbf{Next, 
	we  establish the higher-moment estimates needed for the iteration.}
	Let $(\widetilde p,\widetilde\sigma)$ satisfy
	\eqref{eq:admissible}, with
	\[
	\widetilde p\ge p,\qquad
	q_1:=\widetilde\sigma\widetilde p\ge q.
	\]
	\textbf{We claim that
		\begin{align}
			\iint_{\cQ_0}
			H^2G_{\varepsilon,\widetilde\sigma,+}^{\widetilde p}\dmt
			&\le cD_{\tau}^{q_1+2-q}\Lambda^q,
			\label{eq:gainH}\\
			\iint_{\cQ_0}
			G_{\varepsilon,\widetilde\sigma,+}^{\widetilde p}\dmt
			&\le cD_{\tau}^{q_1-q}\Lambda^q.
			\label{eq:gain}
	\end{align}}

	To prove \eqref{eq:gainH}, 	choose a smooth ambient spacetime cutoff
	$\zeta(x,t)=\varphi(x)\eta(t)$, where
	$\varphi\in C_c^\infty(B_{R_0})$ is equal to one on $B_{r_0}$,
	and $0\le\eta\le1$ vanishes for $t\le\tau/4$ and equals one
	for $t\ge\tau/2$. Set $\chi=\zeta(X,t)$.
	By \eqref{eq:cutoff}, we may arrange that
	\begin{equation}\label{eq:chibounds}
		|\nabla\chi|\le cD_{\tau},\qquad
		|(\partial_t-\Delta)\chi|\le cD_{\tau}^2.
	\end{equation}
	In particular, $\chi=1$ on $\cQ_0$.
	Choose an integer $N\ge1$ such that
	$2N\ge q_1+2-q$, and apply \eqref{eq:damped} with
	$(p,\sigma)$ replaced by
	$(\widetilde p,\widetilde\sigma)$ and $\psi=\chi^N$.
	Since $\chi$ vanishes near $t=0$ and
	\[
	|\chi^N(\partial_t-\Delta)\chi^N|
	+|\nabla(\chi^N)|^2
	\le c(N)D_{\tau}^2\chi^{2N-2},
	\]
	integration in time yields
	\begin{equation}\label{eq:gainstart}
		\iint\chi^{2N}
		G_{\varepsilon,0,+}^{\widetilde p}H^{q_1+2}\dmt
		\le cD_{\tau}^2\iint\chi^{2N-2}
		G_{\varepsilon,0,+}^{\widetilde p}H^{q_1}\dmt.
	\end{equation}
	
	Suppose first that $q_1>q$, and set
	\[
	\xi=\frac{q_1-q}{q_1+2-q}.
	\]
	Then
	\begin{align}
		&\chi^{2N-2}
		G_{\varepsilon,0,+}^{\widetilde p}H^{q_1}
		\notag\\
		&\quad=
		\left(
		\chi^{2N}G_{\varepsilon,0,+}^{\widetilde p}
		H^{q_1+2}
		\right)^\xi
		\left(
		G_{\varepsilon,0,+}^{p}H^q
		\right)^{1-\xi}
		\notag\\
		&\qquad\quad\times
		\chi^{2N(1-\xi)-2}
		G_{\varepsilon,0,+}^{(\widetilde p-p)(1-\xi)}.
		\label{eq:factorization}
	\end{align}
	Since $2N(1-\xi)-2\ge0$ and $\widetilde p\ge p$,
	$\chi^{2N(1-\xi)-2}
		G_{\varepsilon,0,+}^{(\widetilde p-p)(1-\xi)}$ are bounded by
	\eqref{eq:normalized}. Young's inequality therefore gives
	\begin{align}
		cD_{\tau}^2\chi^{2N-2}
		G_{\varepsilon,0,+}^{\widetilde p}H^{q_1}
		&\le
		\frac12\chi^{2N}
		G_{\varepsilon,0,+}^{\widetilde p}H^{q_1+2}
		\notag\\
		&\quad+
		cD_{\tau}^{q_1+2-q}
		G_{\varepsilon,0,+}^{p}H^q.
		\label{eq:younggain}
	\end{align}
	When $q_1=q$, the same conclusion follows directly from
	$\widetilde p\ge p$ and \eqref{eq:normalized}.
	Substituting into \eqref{eq:gainstart} and applying \eqref{eq:low} proves
	\eqref{eq:gainH}.
	
	To obtain \eqref{eq:gain}, we split $\cQ_0$ into
	$\{H\ge D_{\tau}\}$ and $\{H<D_{\tau}\}$. On the former,
	\[
	G_{\varepsilon,\widetilde\sigma,+}^{\widetilde p}
	\le D_{\tau}^{-2}H^2
	G_{\varepsilon,\widetilde\sigma,+}^{\widetilde p},
	\]
	while on the latter,
	\[
	G_{\varepsilon,\widetilde\sigma,+}^{\widetilde p}
	\le cD_{\tau}^{q_1-q}
	G_{\varepsilon,\sigma,+}^{p}.
	\]
	Thus \eqref{eq:gain} follows from
	\eqref{eq:gainH} and \eqref{eq:low}.

	We now fix
	\begin{equation}\label{eq:choosep}
		\rho=n+2,\qquad
		p\ge\max\left\{
		L,\ 2(q+2),\
		\frac{\rho(q+2)^2}{\ell^2}
		\right\},
		\qquad
		\sigma=\frac qp,\qquad
		\sigma'=\frac{q+2}{p}.
	\end{equation}
	The pairs $(p,\sigma)$, $(p\rho,\sigma)$, and
	$(p\rho,\sigma')$ satisfy \eqref{eq:admissible}.
	Consequently, \eqref{eq:gain} gives
	\begin{align}
		\iint_{\cQ_0}
		G_{\varepsilon,\sigma,+}^{p\rho}\dmt
		&\le cD_{\tau}^{q(\rho-1)}\Lambda^q,
		\label{eq:I0}\\
		\iint_{\cQ_0}
		G_{\varepsilon,\sigma',+}^{p\rho}\dmt
		&\le cD_{\tau}^{(q+2)\rho-q}\Lambda^q.
		\label{eq:I2}
	\end{align}

	\smallskip\noindent\emph{\textbf{From $L^2$ to $L^\infty$.}}
	For $k>0$, define
	\begin{equation}\label{eq:vk}
		v_k=(G_{\varepsilon,\sigma}-k)_+^{p/2}.
	\end{equation}
	We use the following notation:
	\begin{equation}\label{eq:levelsets}
		\begin{gathered}
			A_{k,r,s}
			=\{(x,t):s<t<T,\ X(x,t)\in B_r,\ v_k(x,t)>0\},\\
			u(k,r,s)=\iint_{A_{k,r,s}}v_k^2\dmt,\qquad
			a(k,r,s)=\iint_{A_{k,r,s}}\dmt.
		\end{gathered}
	\end{equation}
	For $h>k$, $r>0$, and $0\le s'<T$, we have
	\begin{equation}\label{eq:levelbasic}
		(h-k)^p a(h,r,s')\le u(k,r,s').
	\end{equation}
	
	We next derive a truncated energy estimate.
	Fix $0<r<R$ and $0\le s_0<s_1<T$.
	Choose a smooth ambient spacetime cutoff
	$\zeta(x,t)=\varphi(x)\eta(t)$, where
	$\varphi\in C_c^\infty(B_R)$ and $\eta$ satisfy
	\[
	\begin{gathered}
		0\le\varphi,\eta\le1,\qquad
		\varphi=1\ \text{on }B_r,\\
		\eta=0\ \text{for }t\le s_0,\qquad
		\eta=1\ \text{for }t\ge s_1,
	\end{gathered}
	\]
	with
	\[
	|D\varphi|\le\frac{c}{R-r},\qquad
	|D^2\varphi|\le\frac{c}{(R-r)^2},\qquad
	|\eta'|\le\frac{c}{s_1-s_0}.
	\]
	Set $\psi=\zeta(X,t)$. By \eqref{eq:cutoff},
	\[
	(\partial_t-\Delta)\psi
	=\partial_t\zeta-\operatorname{tr}_{TM_t}D^2\zeta,
	\]
	and hence
	\[
	|\psi(\partial_t-\Delta)\psi|+|\nabla\psi|^2
	\le c(n)\left(
	\frac1{(R-r)^2}+\frac1{s_1-s_0}
	\right).
	\]
	In particular, $\psi$ has compact spatial support,
	vanishes for $t\le s_0$, and equals one on
	$B_r\times[s_1,T)$.
	
	Arguing as in the derivation of \eqref{eq:rawenergy},
	we obtain from \eqref{eq:diff}
	\begin{align}
		\sup_t\int(\psi v_k)^2\,d\mu_t
		&+\iint\left(
		|\nabla(\psi v_k)|^2
		+H^2(\psi v_k)^2
		\right)\dmt\notag\\
		&\le c\iint
		\left(
		|\psi(\partial_t-\Delta)\psi|
		+|\nabla\psi|^2
		\right)v_k^2\dmt\notag\\
		&\quad+c\iint_{\{G_{\varepsilon,\sigma}>k\}}
		\psi^2H^2G_{\varepsilon,\sigma,+}^p\dmt.
		\label{eq:truncated}
	\end{align}
	Indeed, on $\{G_{\varepsilon,\sigma}>k\}$, \eqref{eq:diff} give
\begin{equation}\label{eq:label}
    \begin{aligned}
&(\partial_t-\Delta)
 (G_{\varepsilon,\sigma}-k)^p\\
&\quad\le
 p\sigma|A|^2G_{\varepsilon,\sigma}
 (G_{\varepsilon,\sigma}-k)^{p-1}
 -p(p-1-\gamma^{-1})
 (G_{\varepsilon,\sigma}-k)^{p-2}
 |\nabla G_{\varepsilon,\sigma}|^2.
\end{aligned}
\end{equation}
For sufficiently large $p$, the last term is bounded
above by $-2|\nabla v_k|^2$.
Moreover, since $\sigma p=q$ and
$|A|\le\alpha^{-1}H$, we have
\[
p\sigma|A|^2G_{\varepsilon,\sigma}
(G_{\varepsilon,\sigma}-k)^{p-1}
\le q\alpha^{-2}H^2G_{\varepsilon,\sigma,+}^p.
\]
Multiplying \eqref{eq:label} by $\psi^2$,
integrating with respect to the evolving measure,
and using \eqref{eq:area} and Young's inequality
yields \eqref{eq:truncated}.

	We now apply Stampacchia iteration to obtain a pointwise bound.
	For $j\ge0$, set
	\begin{equation}\label{eq:Qj}
		\begin{gathered}
			r_j=\vartheta R_0+2^{-j-1}(1-\vartheta)R_0,\qquad
			t_j=\tau-2^{-j-1}\tau,\\
			\cQ_j=B_{r_j}\times[t_j,T).
		\end{gathered}
	\end{equation}
	Observe that
	\[
	r_j-r_{j+1}=2^{-j-2}(1-\vartheta)R_0,\qquad
	t_{j+1}-t_j=2^{-j-2}\tau.
	\]
	Choose a smooth ambient spacetime cutoff $\psi_j$ supported
	spatially in $B_{r_j}$, vanishing for $t\le t_j$, and equal
	to one on $\cQ_{j+1}$. As in \eqref{eq:cutoff}, we may arrange
	that
	\begin{equation}\label{eq:psij}
		|\psi_j(\partial_t-\Delta)\psi_j|
		+|\nabla\psi_j|^2
		\le c4^jD_{\tau}^2.
	\end{equation}
	
	Since $\cQ_j\subset\cQ_0$, H\"older's inequality and
	\eqref{eq:I0} give
	\begin{align}
		u(k,r_j,t_j)
		&\le
		\left(
		\iint_{\cQ_0}G_{\varepsilon,\sigma,+}^{p\rho}\dmt
		\right)^{1/\rho}
		a(k,r_j,t_j)^{1-1/\rho}\notag\\
		&\le
		cD_{\tau}^{q(1-1/\rho)}\Lambda^{q/\rho}
		a(k,r_j,t_j)^{1-1/\rho}.
		\label{eq:holder0}
	\end{align}
	Similarly, by \eqref{eq:shift} and \eqref{eq:I2},
	\begin{align}
		&\iint_{A_{k,r_j,t_j}}
		H^2G_{\varepsilon,\sigma,+}^p\dmt\notag\\
		&\quad=
		\iint_{A_{k,r_j,t_j}}
		G_{\varepsilon,\sigma',+}^p\dmt\notag\\
		&\quad\le
		cD_{\tau}^{q+2-q/\rho}\Lambda^{q/\rho}
		a(k,r_j,t_j)^{1-1/\rho}.
		\label{eq:holder2}
	\end{align}
	
	Applying \eqref{eq:truncated} and \eqref{eq:sobolev}
	to $w=\psi_jv_k$, and using \eqref{eq:psij},
	\eqref{eq:holder0}, and \eqref{eq:holder2}, we obtain
	\begin{equation}\label{eq:levelsobolev}
		\left(
		\iint_{\cQ_{j+1}}
		v_k^{2(n+2)/n}\dmt
		\right)^{n/(n+2)}
		\le
		c4^jD_{\tau}^{q(1-1/\rho)+2}
		\Lambda^{q/\rho}
		a(k,r_j,t_j)^{1-1/\rho}.
	\end{equation}
	On the other hand, for $h>k$, H\"older's inequality
	and \eqref{eq:levelbasic} give
	\begin{align}
		(h-k)^p a(h,r_{j+1},t_{j+1})
		&\le u(k,r_{j+1},t_{j+1})\notag\\
		&\le
		a(k,r_j,t_j)^{2/(n+2)}
		\left(
		\iint_{\cQ_{j+1}}
		v_k^{2(n+2)/n}\dmt
		\right)^{n/(n+2)}.
		\label{eq:levelmeasure}
	\end{align}
	Consequently,
	\begin{equation}\label{eq:levelrecurrence}
		\begin{aligned}
			a(h,r_{j+1},t_{j+1})
			&\le
			c4^j(h-k)^{-p}
			D_{\tau}^{q(1-1/\rho)+2}\Lambda^{q/\rho}
			a(k,r_j,t_j)^{\gamma_*},\\
			\gamma_*
			&=1+\frac2{n+2}-\frac1\rho
			=1+\frac1{n+2}>1.
		\end{aligned}
	\end{equation}
	
	If $\Lambda=0$, then \eqref{eq:low} and continuity
	imply $G_{\varepsilon,\sigma,+}\equiv0$ on the
	smaller cylinder. Hence we may assume $\Lambda>0$.
	
	Set
	\[
	a(k,j):=a(k,r_j,t_j),\qquad k_0=D_{\tau}^\sigma.
	\]
	Since the cylinders are nested, $a(k,j)$ is
	nonincreasing in both $k$ and $j$. Moreover,
	Chebyshev's inequality and \eqref{eq:low} yield
	\begin{equation}\label{eq:a0}
		a(k_0,0)
		\le k_0^{-p}
		\iint_{\cQ_0}G_{\varepsilon,\sigma,+}^p\dmt
		\le cD_{\tau}^{-q}\Lambda^q,
	\end{equation}
	where $\sigma p=q$.
	
	We may therefore apply Lemma~\ref{lem:stampacchia}
	to \eqref{eq:levelrecurrence}, with
	\[
	\omega=\gamma_*-1=\frac1{n+2},
	\qquad
	A=cD_{\tau}^{q(1-1/\rho)+2}\Lambda^{q/\rho}.
	\]
	The lemma gives
	\begin{equation}\label{eq:stampacchia-zero}
		a(k_0+d,\vartheta R_0,\tau)=0,
	\end{equation}
	provided
	\[
	\begin{aligned}
		d^p
		&\ge cA\bigl(D_{\tau}^{-q}\Lambda^q\bigr)^\omega\\
		&=cD_{\tau}^{q+2-2q/(n+2)}
		\Lambda^{2q/(n+2)}.
	\end{aligned}
	\]
	Thus we may take
	\begin{equation}\label{eq:d}
		d=cD_{\tau}^\sigma
		\left(
		D_{\tau}^{n+2-q}\Lambda^q
		\right)^{2/(p(n+2))}.
	\end{equation}
	By continuity of $G_{\varepsilon,\sigma,+}$,
	\eqref{eq:stampacchia-zero} implies
	\begin{equation}\label{eq:Gsigmabound}
		G_{\varepsilon,\sigma,+}
		\le
		D_{\tau}^\sigma
		\left[
		1+c\left(
		D_{\tau}^{n+2-q}\Lambda^q
		\right)^{2/(p(n+2))}
		\right]
	\end{equation}
	on $B_{\vartheta R_0}\times[\tau,T)$.
	
	Recalling \eqref{eq:Geps} and using
	$\Lambda^q\le100\cE_q$, we obtain
	\begin{equation}\label{eq:sublinearG}
		G
		\le\varepsilon H+
		cD_{\tau}^\sigma
		\left[
		1+
		\left(D_{\tau}^{n+2-q}\cE_q\right)^{2/(p(n+2))}
		\right]H^{1-\sigma}.
	\end{equation}
	Young's inequality and $\sigma p=q$ then give
	\begin{equation}\label{eq:Gfinal}
		G
		\le2\varepsilon H+
		c(n,\alpha,q,\varepsilon)D_{\tau}
		\left[
		1+
		\left(D_{\tau}^{n+2-q}\cE_q
		\right)^{2/(q(n+2))}
		\right].
	\end{equation}
	
\end{proof}

Next, we give the proof of Corollary~\ref{cor:collar}.
\begin{proof} [Proof of Corollary~\ref{cor:collar}]
	
	By translation and parabolic scaling, we may assume
	$z_0=(0,0)$ and $R=1$. Set
	\[
	\mathcal C=P_2\setminus P_\theta,
	\qquad
	\mathcal J=\iint_{\mathcal C}H^q\dmt\le K.
	\]
	Choose
	\[
	b_\theta=\frac{2+\theta}{2},\qquad
	h_\theta=\frac{4-\theta^2}{4},\qquad
	I_\theta=[-4+h_\theta,-4+2h_\theta].
	\]
	Since $1<\theta<2$, we have
	$\theta<b_\theta<2$ and
	$I_\theta\Subset(-4,-\theta^2)$.
	In particular,
	$
	B_{b_\theta}\times I_\theta\subset\mathcal C.
	$
	By the mean-value inequality for integrals, there exists
	$s\in I_\theta$ such that
	\begin{equation}\label{eq:goodSlice}
		\int_{M_s\cap B_{b_\theta}}H^q\,d\mu_s
		\le h_\theta^{-1}\mathcal J.
	\end{equation}
	Moreover, since
	\[
	(B_{b_\theta}\setminus B_\theta)
	\times[s,0]\subset\mathcal C,
	\]
	we have
	\begin{equation}\label{eq:goodAnnulus}
		\int_s^0
		\int_{M_t\cap(B_{b_\theta}\setminus B_\theta)}
		H^q\dmt
		\le\mathcal J.
	\end{equation}
	
	We apply Theorem~\ref{thm:main} to the time-translated
	flow $\widetilde X(x,t)=X(x,s+t)$, with
	\begin{equation}\label{eq:goodParameters}
		\widehat R=\sqrt{-2ns},\qquad
		\lambda=\frac{b_\theta}{\widehat R},\qquad
		\delta=\frac{b_\theta-\theta}{\widehat R},\qquad
		\vartheta=\frac1\theta,\qquad
		\tau=\frac1{4n}.
	\end{equation}
	The corresponding time interval has length
	$\widehat R^2/(2n)=-s$.
	Furthermore,
	\begin{equation}\label{eq:goodTarget}
		\vartheta(\lambda-\delta)\widehat R=1,
		\qquad
		s+\tau\widehat R^2=\frac{s}{2}
		\le-1-\frac{\theta^2}{4}<-1.
	\end{equation}
	
	Let $\widehat{\cE}_q$ be the integral quantity
	\eqref{eq:E} corresponding to the time-translated flow,
	with scale $\widehat R$ and parameters
	$\lambda,\delta$ given by \eqref{eq:goodParameters}.
	By \eqref{eq:goodSlice} and \eqref{eq:goodAnnulus},
	\begin{align}
		\widehat{\cE}_q
		&\le
		\widehat R^{q-n-2}
		\left[
		(-s)h_\theta^{-1}
		+\frac{\widehat R^2}{(b_\theta-\theta)^2}
		\right]\mathcal J\notag\\
		&\le C(n,q,\theta)\mathcal J,
		\label{eq:goodBudget}
	\end{align}
	where we used
	$
	\sqrt{5n}\le\widehat R\le\sqrt{8n}.
	$
	Furthermore, we have
	\begin{equation}\label{eq:goodD}
		\widehat D_{\tau}
		=1+2\sqrt n+\frac{\widehat R}{\theta-1}
		\le C(n,\theta).
	\end{equation}
	Consequently, Theorem~\ref{thm:main} yields Corollary \ref{cor:collar}.
	
\end{proof}

\section{Applications to ancient flows}\label{sec:areaEntropy}

In this section we prove Theorem~\ref{thm:areaAncient}. We begin with the following estimate.

\begin{lemma}\label{lem:areaDissipation}
	Let a smooth hypersurface mean curvature flow be properly defined in
	$P_{3R}(z_0)$, where $z_0=(x_0,t_0)$ and $R>0$. If
	$\iint_{P_{3R}(z_0)}\dmt<\infty$, then
	\begin{equation}\label{eq:areaDissipation}
		\iint_{P_{2R}(z_0)}H^2\dmt
		\le \frac{C(n)}{R^2}\iint_{P_{3R}(z_0)}\dmt.
	\end{equation}
	
\end{lemma}
\begin{proof}
	
	Choose a smooth spatial cutoff
	$\chi\in C_c^\infty(B_{3R}(x_0))$ such that
	\[
	0\le\chi\le1,\qquad
	\chi=1\ \text{on }B_{2R}(x_0),\qquad
	|D\chi|\le\frac{c}{R}.
	\]
	Let $\eta$ be a smooth nondecreasing time cutoff satisfying
	\[
	\begin{gathered}
		0\le\eta\le1,\qquad
		\eta=0\ \text{for }t\le t_0-8R^2,\\
		\eta=1\ \text{for }t\ge t_0-4R^2,\qquad
		0\le\eta'\le\frac{c}{R^2}.
	\end{gathered}
	\]
	Set
	\[
	\mathcal A(t)
	=\int_{M_t}\eta(t)\chi(X)^2\,d\mu_t.
	\]
	Using $\partial_tX=-H\nu$ and the area evolution
	\eqref{eq:area}, we compute
	\begin{align}
		\mathcal A'(t)
		&=\int_{M_t}\eta'\chi^2\,d\mu_t
		-2\int_{M_t}\eta\chi H
		\langle D\chi,\nu\rangle\,d\mu_t
		\notag\\
		&\quad-\int_{M_t}\eta\chi^2H^2\,d\mu_t.
		\label{eq:weightedAreaVariation}
	\end{align}
	Since
	\[
	2\chi|H||D\chi|
	\le\frac12\chi^2H^2+2|D\chi|^2,
	\]
	we obtain
	\begin{equation}\label{eq:areaDamped}
		\mathcal A'(t)
		+\frac12\int_{M_t}\eta\chi^2H^2\,d\mu_t
		\le
		\int_{M_t}
		\left(\eta'\chi^2+2\eta|D\chi|^2\right)d\mu_t.
	\end{equation}
	Integrating from $t_0-8R^2$ to $t_0$, and using
	$\mathcal A(t_0-8R^2)=0$ and $\mathcal A(t_0)\ge0$,
	we find
	\[
	\frac12\int_{t_0-8R^2}^{t_0}
	\int_{M_t}\eta\chi^2H^2\,d\mu_t\,dt
	\le
	\frac{c}{R^2}
	\iint_{P_{3R}(z_0)}\dmt.
	\]
	Since $\eta\chi^2=1$ on $P_{2R}(z_0)$,
	it follows that
	\[
	\iint_{P_{2R}(z_0)}H^2\dmt
	\le\frac{C(n)}{R^2}
	\iint_{P_{3R}(z_0)}\dmt,
	\]
	as claimed. 
	
\end{proof}

The area estimate and Corollary~\ref{cor:collar} give a local
pinching estimate in terms of spacetime area alone.

\begin{proposition}
	\label{prop:areaPinching}
	Let $X$ be a smooth mean curvature flow properly defined
	in $P_{3R}(z_0)$, satisfying
	\[
	\kappa_1+\cdots+\kappa_{m+1}\ge\alpha|A|
	\]
	throughout $P_{2R}(z_0)$. Define 
	\begin{equation}\label{eq:spacetimeArea}
		\mathcal V(z_0;R)
		:=R^{-n-2}\iint_{P_{3R}(z_0)}\dmt.
	\end{equation}
	
	If $\mathcal V(z_0;R)<\infty$, then for every
	$\varepsilon>0$,
	\begin{equation}\label{eq:areaPinching}
		\begin{aligned}
			\kappa_1
			&\ge-\varepsilon H
			-\frac{C}{R}
			\left(1+\mathcal V(z_0;R)^{1/(n+2)}\right),\\
			\sum_{j=m+1}^{n}(\kappa_n-\kappa_j)
			&\le\sum_{j=1}^{m}\kappa_j+\varepsilon H
			+\frac{C}{R}
			\left(1+\mathcal V(z_0;R)^{1/(n+2)}\right)
		\end{aligned}
	\end{equation}
	on $P_R(z_0)$, where $C=C(n,\alpha,\varepsilon)$.
\end{proposition}

\begin{proof}
	Since $\mathcal C_{3/2,R}(z_0)\subset P_{2R}(z_0)$,
	Lemma~\ref{lem:areaDissipation} yields
	\begin{align}
		\cC_{2,3/2}(z_0;R)
		&=R^{-n}\iint_{\mathcal C_{3/2,R}(z_0)}
		H^2\dmt\notag\\
		&\le C(n)R^{-n-2}
		\iint_{P_{3R}(z_0)}\dmt\notag\\
		&=C(n)\mathcal V(z_0;R).
		\label{eq:areaToCollar}
	\end{align}
	Applying Corollary~\ref{cor:collar} with
	$q=2$ and $\theta=3/2$, so that the exponent
	$2/[q(n+2)]=1/(n+2)$, gives \eqref{eq:areaPinching}
	after adjusting the constant.
\end{proof}

We apply this estimate at arbitrarily large scales to prove Theorem~\ref{thm:areaAncient}.

\begin{proof}[Proof of Theorem~\ref{thm:areaAncient}]
	
	Fix $p\in M$ and $t_0<0$, and set
	$z_0=(X(p,t_0),t_0)$. Since the flow is ancient,
	we may apply Proposition~\ref{prop:areaPinching}
	at arbitrarily large scales. By
	\eqref{eq:uniformAreaGrowth},
	\begin{equation}\label{eq:areaBudgetAncient}
		\begin{aligned}
			\mathcal V(z_0;R)
			&\le R^{-n-2}
			\int_{t_0-9R^2}^{t_0}\Lambda(3R)^n\,dt\\
			&=\cdot3^{n+2}\Lambda
		\end{aligned}
	\end{equation}
	for every $R>0$.
	Thus, for each $\varepsilon>0$,
	Proposition~\ref{prop:areaPinching} gives
	\begin{align*}
		\kappa_1(p,t_0)
		&\ge-\varepsilon H(p,t_0)
		-\frac{C}{R}
		\left(1+\Lambda^{1/(n+2)}\right),\\
		(n-m)\kappa_n(p,t_0)-H(p,t_0)
		&\le\varepsilon H(p,t_0)
		+\frac{C}{R}
		\left(1+\Lambda^{1/(n+2)}\right),
	\end{align*}
	where $C=C(n,\alpha,\varepsilon)$ is independent
	of $R$. Letting $R\to\infty$ and then
	$\varepsilon\downarrow0$, we obtain
	\[
	\kappa_1(p,t_0)\ge0,\qquad
	(n-m)\kappa_n(p,t_0)\le H(p,t_0).
	\]
	Since $(p,t_0)$ was arbitrary, these inequalities
	hold throughout the ancient flow. 
	This proves \eqref{eq:areaAncientCone}.
\end{proof}

\section{Local planarity in arbitrary codimension}\label{sec:planarity}

We prove Theorems~\ref{hc:thm:local} and~\ref{hc:cor:timePlanarity} and
Corollary~\ref{hc:thm:ancient} using the normal-valued notation
introduced above.

Fix the constants and the function $f$ from Lemma~\ref{hc:lem:input}.
The following calculation is contained in \cite[proof of Theorem~1.1,
equations~(2.6)--(2.9)]{LNZ25}.

\begin{lemma}\cite{LNZ25}\label{hc:lem:absorbing}
	 Let $F:M^n\times I\to\R^N$, $n\ge2$ and
	$N\ge n+1$, be a smooth mean curvature flow satisfying
\eqref{hc:eq:pinch} on the region under consideration.
On $\{f>0\}$ define
\begin{equation}\label{hc:eq:uw}
 u=\frac{|\Ah|^2}{f^{1-\sigma_0}},\qquad w=u^{1/\sigma_0},
\end{equation}
where $f$, $c_0$ and $\delta_0$ are defined in Lemma \ref{hc:lem:input}.
Extend $u$ and $w$ by zero on $\{f=0\}$. These extensions are continuous
and nonnegative. On $\{w>0\}$, the function $w$ is smooth and satisfies
\begin{equation}\label{hc:eq:absorbing}
 (\partial_t-\Delta) w\le-\beta w^2,
 \qquad \beta=\beta(n,\delta_0)>0.
\end{equation}
\end{lemma}

\begin{proof}
By \eqref{hc:eq:fBounds},
\[
 |\Ah|^2\le K f,\qquad K=\frac{c_0}{c_0-c_n+\delta_0}.
\]
Hence
\begin{equation}\label{hc:eq:extension}
 0\le u\le Kf^{\sigma_0},\qquad
 0\le w\le K^{1/\sigma_0}f.
\end{equation}
This proves continuity for $u$ and $w$ at $f=0$. At a point where $w>0$, both $|\Ah|^2$ and
$f$ are positive.
We calculate that
\begin{align}
 (\partial_t-\Delta)u
 &=f^{\sigma_0-1}(\partial_t-\Delta)|\Ah|^2
   +(\sigma_0-1)\frac{u}{f}(\partial_t-\Delta)f\notag\\
 &\quad+2(1-\sigma_0)\langle\nabla u,\nabla\log f\rangle
   -\sigma_0(1-\sigma_0)u|\nabla\log f|^2.
 \label{hc:eq:quotient}
\end{align}
By \eqref{hc:eq:La},
\begin{align*}
 (\partial_t-\Delta)u
 &\le-\sigma_0\frac{u}{f}(\partial_t-\Delta)f
      +2(1-\sigma_0)\langle\nabla u,\nabla\log f\rangle\\
 &\quad-\sigma_0(1-\sigma_0)u|\nabla\log f|^2.
\end{align*}
It follows that
\begin{equation}\label{hc:eq:square}
 \begin{aligned}
 (\partial_t-\Delta)w
 &\le-\frac{w}{f}(\partial_t-\Delta)f\\
 &\quad-\frac{1-\sigma_0}{\sigma_0^2}u^{1/\sigma_0-2}
       |\nabla u-\sigma_0u\nabla\log f|^2\\
 &\le-\frac{w}{f}(\partial_t-\Delta)f.
 \end{aligned}
\end{equation}
Moreover $|h|^2\ge|\HH|^2/n$, so \eqref{hc:eq:Lf},
\eqref{hc:eq:fBounds} and \eqref{hc:eq:extension} give
\[
 \frac{(\partial_t-\Delta) f}{f}\ge2|h|^2\ge\frac2n|\HH|^2
 \ge\frac{2}{nc_0}f
 \ge\frac{2}{nc_0K^{1/\sigma_0}}w.
\]
Thus \eqref{hc:eq:absorbing} holds with
$\beta=2/(nc_0K^{1/\sigma_0})$.
\end{proof}

\begin{proof}[Proof of Theorem~\ref{hc:thm:local}]
The assertion is immediate when $N=n+1$, since $|\Ah|^2=0$. For
$N\ge n+2$, we consider the function $w$ from Lemma~\ref{hc:lem:absorbing}.
Write $a_0=t_0-R^2$ and choose $\varphi\in C_c^\infty(B_R(y))$ with
\[
 0\le\varphi\le1,\qquad \varphi=1\ \text{on }B_{R/2}(y),\qquad
 |D\varphi|\le C R^{-1},\quad |D^2\varphi|_{\mathrm{op}}\le C R^{-2}.
\]
Choose a smooth nondecreasing function $\vartheta:\R\to[0,1]$ satisfying
\[
 \vartheta=0\ \text{on }(-\infty,a_0],\qquad
 \vartheta=1\ \text{on }[a_0+R^2/2,\infty),\qquad
 |\vartheta'|\le C R^{-2}.
\]
Set $\psi=\varphi(F)\vartheta(t)$.
For a smooth ambient function $\zeta(x,t)$ and its pullback
$\psi(p,t)=\zeta(F(p,t),t)$, the identities $\partial_tF=\Delta F=\HH$
give
\begin{equation}\label{hc:eq:cutoff}
	(\partial_t-\Delta)\psi
	=\bigl(\partial_t\zeta-\tr_{TM_t}D_x^2\zeta\bigr)(F,t).
\end{equation}
Hence,
\begin{equation}\label{hc:eq:cutoffBound}
 |\psi(\partial_t-\Delta)\psi|+|\nabla\psi|^2\le C(n)R^{-2}.
\end{equation}
Let the function $W=\psi^2w$. Properness on the closed cylinder makes $W$'s support on
$[a_0,t_0]$ compact. Notice that  $W=0$ at $t=a_0$. If $W$'s maximum is zero, the estimate
follows. Otherwise it has a positive maximum at some $(p_*,t_*)$ with
$t_*>a_0$, $\psi>0$ and $w>0$. At that point,
\[
 \nabla w=-2w\frac{\nabla\psi}{\psi},\qquad (\partial_t-\Delta)W\ge0.
\]
 By \eqref{hc:eq:absorbing},
\begin{align}
 0\le(\partial_t-\Delta)W
 &\le-\beta\psi^2w^2+2\psi w(\partial_t-\Delta)\psi+6w|\nabla\psi|^2
 \notag\\
 &\le-\beta\psi^2w^2+C(n)R^{-2}w.
 \label{hc:eq:maxcalc}
\end{align}
It follows that $\max W\le C(n)/(\beta R^2)$.
On $B_{R/2}(y)\times[t_0-R^2/2,t_0]$ one has $\psi=1$, and therefore
\[
 w\le C(n,\delta_0)R^{-2}.
\]
At $f=0$ the function $w$ vanishes. Since $|\Ah|^2=w^{\sigma_0} f^{1-\sigma_0}$ on $\{f>0\}$ and
$f\le c_0|\HH|^2$, this proves \eqref{hc:eq:local}. 
\end{proof}

\begin{proof}[Proof of Theorem~\ref{hc:cor:timePlanarity}]
Fix $p\in M$ and $0<t<T$, and set $y=F(p,t)$ and $R=\sqrt{t/2}$.
The cylinder $\overline{B_R(y)}\times[t/2,t]$ satisfies the hypotheses
of Theorem~\ref{hc:thm:local}. Applying that theorem at $(p,t)$ gives
\begin{equation}\label{hc:eq:timeFromLocal}
 \begin{aligned}
  |\Ah|^2(p,t)
  &\le C R^{-2\sigma_0}|\HH|^{2-2\sigma_0}(p,t)\\
  &=2^{\sigma_0}C t^{-\sigma_0}|\HH|^{2-2\sigma_0}(p,t).
 \end{aligned}
\end{equation}
Absorbing $2^{\sigma_0}$ into $C$ proves \eqref{hc:eq:timePlanarity}
for $t\le T$.

By Lemma~\ref{hc:lem:input}, $f\ge0$ and $(\partial_t-\Delta)f\ge0$.
If $|\HH|(p_*,t_*)=0$ for some $t_*>0$, the strong minimum principle
gives $f=0$ on $M\times[0,T]$, hence $\Avec=\HH=0$. 
\end{proof}

\begin{proof}[Proof of Corollary~\ref{hc:thm:ancient}]
If $\Avec\equiv0$, then $\HH\equiv0$, then the submanifold is a static affine $n$-plane. Otherwise  $|\HH|>0$ at any time. Applying \eqref{hc:eq:timePlanarity} with initial time $s<t$
and letting $s\to-\infty$ gives $\Ah\equiv0$.
Put $c=c_n-\delta_0<1$. Since
$|h|\le|\Avec|\le\sqrt c\,|\HH|$,
\begin{equation}\label{hc:eq:naffCriterion}
 |\HH|g-h\ge(1-\sqrt c)|\HH|g>0.
\end{equation}
Thus \cite[Proposition~2.5]{Naff22} (also see
\cite[Section~1.1]{LNZ25}) places each slice in an affine
$(n+1)$-plane.
Its affine span cannot have dimension $n$, as that would give
$\HH=0$, so its dimension is exactly $n+1$; denote it by $\Pi_t$.

Notice that in the spacetime-proper setting, affine containment is preserved:
\begin{equation}\label{hc:eq:affinePreservation}
	F(M,s)\subset y+E
	\quad\Longrightarrow\quad
	F(M,t)\subset y+E,\qquad s\le t\le b,
\end{equation}
where $E\subset\R^N$ is a linear subspace. Indeed, for any
unit vector $e_0\in E^\perp$, set
\[
v=\langle F-y,e_0\rangle,\qquad
\Phi=1+|F-y|^2+2n(t-s).
\]
Since $(\partial_t-\Delta)v=(\partial_t-\Delta)\Phi=0$,
the functions
\[
Z_\pm=\pm v-\epsilon\Phi-\kappa(t-s),
\qquad \epsilon,\kappa>0,
\]
satisfy
\[
(\partial_t-\Delta)Z_\pm=-\kappa,
\qquad Z_\pm(\cdot,s)<0.
\]
 Notice that $
Z_\pm\le |F-y|-\epsilon(1+|F-y|^2)
$ , we can use the  maximum principle (c.f. Theorem 1.2 in \cite{Ma18}) to get $Z_\pm\le0$.
Letting $\kappa\downarrow0$ and $\epsilon\downarrow0$
gives $v=0$, proving \eqref{hc:eq:affinePreservation}.
This proving the corollary.
\end{proof}

\end{document}